\documentclass[11pt]{article}
\usepackage{amsmath}
\usepackage{amssymb}
\usepackage{amsthm}
\usepackage[usenames]{color}
\usepackage{amscd}
\usepackage{dsfont}
\usepackage{indentfirst}

\usepackage[colorlinks=true,linkcolor=blue,filecolor=red,
citecolor=webgreen]{hyperref}
\definecolor{webgreen}{rgb}{0,.5,0}
\definecolor{amber}{rgb}{1.0, 0.49, 0.0}
\numberwithin{equation}{section}

\def\C{{\mathds{C}}}

\def\N{{\mathds{N}}}
\def\Z{{\mathds{Z}}}
\def\P{{\mathds{P}}}

\def\1{{\bf 1}}

\def\Real{\operatorname{Re}}
\def\Imag{\operatorname{Im}}

\newcommand{\DOT}{\text{\rm\Huge{.}}}

\newtheorem{theorem}{Theorem}[section]

\newtheorem{corollary}[theorem]{Corollary}

\begin{document}
	
	\title{{\bf 
			Generalized cyclotomic polynomials  
			associated with regular systems of divisors and arbitrary sets of positive 
			integers}}
	\author{L\'aszl\'o T\'oth \\ \\ Department of Mathematics, University of P\'ecs \\
		Ifj\'us\'ag \'utja 6, 7624 P\'ecs, Hungary \\ E-mail: {\tt ltoth@gamma.ttk.pte.hu}}
	\date{}
	\maketitle
	
	\begin{abstract} 
		We introduce and study the generalized cyclotomic polynomials $\Phi_{A,S,n}(x)$ 
		associated with a regular system $A$ of divisors and an arbitrary set $S$ of positive 
		integers. We show that all of these polynomials have integer coefficients, they can be expressed as the product
		of certain classical cyclotomic polynomials $\Phi_d(x)$ with $d\mid n$, and enjoy many other 
		properties which are similar to the classical and unitary cases. 
		We also point out some related Menon-type identities. One of them seems to be new even for the 
		cyclotomic polynomials $\Phi_n(x)$. 
	\end{abstract}
	
	

	\section{Introduction}
	
	The cyclotomic polynomials $\Phi_n(x)$ are defined by
	\begin{equation*} 
		\Phi_n(x) = \prod_{\substack{j=1\\ (j,n)=1}}^n \left(x- \zeta_n^j \right),
	\end{equation*}
	where $\zeta_n:= e^{2\pi i/n}$, $n\in \N:=\{1,2,\ldots\}$. 
	
	We recall that a divisor $d$ of $n$ ($d,n\in \N$) is a unitary divisor if $(d,n/d)=1$,
	notation $d\mid \mid n$. The unitary cyclotomic polynomials $\Phi_n^*(x)$ are defined using unitary divisors, as follows:
	\begin{equation*} 
		\Phi_n^*(x)= \prod_{\substack{j=1\\ (j,n)_*=1}}^n \left(x- \zeta_n^j \right),
	\end{equation*}
	where $(j,n)_* = \max \{d\in \N: d\mid j, d\mid \mid n\}$. These polynomials have properties similar
	to the classical cyclotomic polynomials. For example, for every $n\in \N$,
	\begin{equation} \label{x_n_unit}
		x^n-1 = \prod_{d\mid \mid n} \Phi_d^*(x),    
	\end{equation}
	and 
	\begin{equation*}
		\Phi_n^*(x)= \prod_{d\mid \mid n} \left(x^d-1 \right)^{\mu^*(n/d)},    
	\end{equation*}
	where $\mu^*(n)= (-1)^{\omega(n)}$ is the unitary M\"{o}bius function, $\omega(n)$ denoting the number of distinct prime divisors of $n$. Also, the unitary cyclotomic polynomials are in the class of inclusion-exclusion polynomials. 
	See Jones et al. \cite{JonTotetal2020}, Moree and the author \cite{MorTot2020} for a detailed discussion of these and some related properties. Also see Bachman \cite{Bac2021}.
	
	However, not all polynomials $\Phi_n^*(x)$ are irreducible over the rationals. More exactly, for every $n\in \N$ one has the irreducible factorization 
	\begin{equation} \label{Phi_star_prod}
		\Phi_n^*(x)= \prod_{\substack{d\mid n\\ \kappa(d)=\kappa(n)}} \Phi_d(x),  
	\end{equation}
	where $\kappa(n)= \prod_{p\mid n} p$ is the square-free kernel of $n$. 
	See \cite[Th.\ 2]{MorTot2020}. It follows from \eqref{Phi_star_prod} that $\Phi_n^*(x)$ have integer 
	coefficients. 
	
	The polynomials 
	\begin{equation} \label{def_Q_n}
		Q_n(x) = \prod_{\substack{j=1\\ (j,n)\text { a square}}}^n \left(x- \zeta_n^j \right),
	\end{equation}
	have been discussed by Sivaramakrishnan \cite[Sect.\ X.5]{Siv1989}. For every $n\in \N$ we have 
	the identities
	\begin{equation} \label{x_n_square}
		x^n-1 = \prod_{\substack{d\mid n\\ n/d \text{ squarefree}}} Q_d(x),    
	\end{equation}
	\begin{equation*}
		Q_n(x)= \prod_{d\mid n} \left(x^d-1 \right)^{\lambda(n/d)},    
	\end{equation*}
	where $\lambda(n) =(-1)^{\Omega(n)}$ is the Liouville function, $\Omega(n)$ denoting the number of distinct prime power divisors of $n$. At the same time,
	\begin{equation*} 
		Q_n(x)= \prod_{d^2\mid n} \Phi_{n/d^2}(x), 
	\end{equation*}
	hence $Q_n(x)$ have integer coefficients as well. 
	
	The inverse cyclotomic polynomials are defined by
	\begin{equation*}
		\Psi_n(x)=\prod_{\substack{j=1\\ (j,n)>1}}^n \left(x- \zeta_n^j \right),
	\end{equation*}
	see Moree \cite{Mor2009}. One has
	\begin{equation*}
		\Psi_n(x)= \frac{x^n-1}{\Phi_n(x)}= \prod_{\substack{d\mid n\\ d<n}} \Phi_d(x).  
	\end{equation*}
	
	It is the goal of the present paper to investigate a common generalization of all of the above polynomials.
	Let $A = (A(n))_{n\in \N}$ be a regular system of divisors. See Section \ref{Sect_regular} for its definition and some basic properties. In particular, $D=(D(n))_{n\in \N}$ and $U=(U(n))_{n\in \N}$ are regular, where $D(n)$ is the set
	of all divisors of $n$ and $U(n)$ is the set of unitary divisors of $n$. Furthermore, let $S$ be an arbitrary nonempty subset of $\N$. We define the cyclotomic polynomials $\Phi_{A,S,n}(x)$ by
	\begin{equation} \label{Phi_A_S}
		\Phi_{A,S,n}(x)= \prod_{\substack{j=1\\ (j,n)_A\in S}}^n \left(x- \zeta_n^j \right),
	\end{equation}
	where  
	\begin{equation} \label{(j,n)_A}
		(j,n)_A = \max \{d\in \N: d\mid j, d\in A(n)\}. 
	\end{equation}
	
	If $S=\{1\}$ and $A$ is a regular system of divisors, then \eqref{Phi_A_S} reduces to the polynomials 
	\begin{equation} \label{Phi_A}
		\Phi_{A,n}(x) :=\Phi_{A,\{1\},n}(x) = \prod_{\substack{j=1\\ (j,n)_A=1}}^n \left(x- \zeta_n^j \right),
	\end{equation}
	introduced by Nageswara Rao \cite{Nag1976}. Here \eqref{Phi_A} recovers $\Phi_n(x)$ and $\Phi_n^*(x)$ if 
	$A(n)=D(n)$ and $A(n)=U(n)$, respectively.
	If $A(n)=D(n)$ and $S$ is arbitrary, then \eqref{Phi_A_S} are the polynomials investigated 
	by the author \cite{Tot1990}, which recover the polynomials $Q_n(x)$ and $\Psi_n(x)$ if $S$ is the set of squares
	and $S=\N\setminus \{1\}$, respectively.
	
	Note that generalizations of arithmetic functions, in particular Euler's function and Ramanujan sums, associated to regular systems of divisors and arbitrary sets are known in the literature, see Section \ref{Sect_func}. However, the corresponding cyclotomic
	polynomials have not been considered, as far as we know. Further generalizations can also be studied. Given a 
	regular system $A$ of divisors and $k\in \N$ one can consider the system $A_k=(A_k(n))_{n\in \N}$ with
	$A_k(n)=\{d\in \N: d^k\in A(n^k)\}$ and the cyclotomics defined with respect to $A_k$. See Section \ref{Sect_regular} 
	for some more details. However, we confine ourselves with the previous generalizations.
	
	We show that all polynomials $\Phi_{A,S,n}(x)$ have integer coefficients and they can be expressed as the product
	of certain polynomials $\Phi_d(x)$ with $d\mid n$. We also point out some other properties and identities 
	concerning the polynomials $\Phi_{A,S,n}(x)$ and $\Phi_{A,n}(x)$. One of them seems to be new even for the classical
	cyclotomic polynomials. Namely, let $n\in \N$ and let $\chi$ be an arbitrary Dirichlet character (mod $n$) with conductor 
	$d$ ($d\mid n$). Then for real $x>1$ (or formally),
	\begin{equation} \label{Menon_D_new}
		\prod_{j=1}^n \left(x^{(j-1,n)}-1\right)^{\Real(\chi(j))} = \prod_{\delta \mid n/d} \Phi_{d\delta}(x)^{\varphi(n)/\varphi(d\delta)}, 
	\end{equation}
	where $\varphi$ is Euler's function. If $\chi$ is a primitive character (mod $n$), then $d=n$, and \eqref{Menon_D_new}
	gives 
	\begin{equation} \label{Menon_D_new_primitive}
		\Phi_n(x)= \prod_{j=1}^n \left(x^{(j-1,n)}-1\right)^{\Real(\chi(j))}.
	\end{equation}
	
	Note that the products in \eqref{Menon_D_new} and \eqref{Menon_D_new_primitive} are, in fact, over $j$ with $(j,n)=1$,
	since $\chi(j)=0$ for $(j,n)>1$. These are Menon-type identities, involving $(j-1,n)$, where $j$ runs over a reduced residue system (mod $n$). The original Menon identity reads
	\begin{equation} \label{Menon_id_origi}
		\sum_{\substack{j=1\\ (j,n)=1}}^n (j-1,n) = \tau(n)\varphi(n) \quad (n\in \N),
	\end{equation}
	where $\tau(n)$ is the number of divisors of $n$. See the survey by the author \cite{Tot2023}.
	
	We also deduce generalizations of the M\"oller-Endo identities and the Grytczuk-Tropak recursion formula
	concerning the coefficients of the discussed cyclotomic polynomials. The main results are included in Section 
	\ref{Section_Results} and their proofs are given in Section \ref{Section_Proofs}.
	
	For an overview of properties of the classical cyclotomic polynomials, in particular those discussed and generalized
	in the present paper, we refer to Gallot et al. \cite{GalMorHom2011}, Herrera-Poyatos and Moree \cite{HerMor2021}, Sanna \cite{San2022}.
	
	\section{Preliminaries}
	
	\subsection{Regular systems of divisors} \label{Sect_regular}
	Let $A(n)$ be a subset of the set of positive divisors of $n$ for every $n\in \N$. 
	The $A$-convolution of the functions
	$f, g : \N \to \C$ is defined by
	\begin{equation} \label{A_convo}
		(f *_A g)(n) = \sum_{d\in A(n)} f(d)g(n/d) \quad (n\in  \N).
	\end{equation}
	
	Let ${\cal A}$ denote the set of arithmetic functions $f:\N \to \C$. The convolution \eqref{A_convo} and the system $A = (A(n))_{n\in \N}$ of divisors are called regular, cf. Narkiewicz \cite{Nar1963}, if the following conditions hold true:
	
	(a) $({\cal A},+, *_A)$ is a commutative ring with unity,
	
	(b) the $A$-convolution of multiplicative functions is multiplicative,
	
	(c) the constant $1$ function has an inverse $\mu_A$ (generalized M\"{o}bius function)
	with respect to $*_A$ and $\mu_A(p^a)\in  \{-1, 0\}$ for every prime power $p^a$ ($a \ge  1$).
	
	It can be shown that the system $A = (A(n))_{n\in \N}$ is regular if and only if
	
	(i) $A(mn) = \{de : d \in A(m), e\in  A(n)\}$ for every $m, n\in \N$ with $(m, n) = 1$,
	
	(ii) for every prime power $p^a$ ($a \ge 1$) there exists a divisor $t = t_A(p^a)$ of $a$, called
	the type of $p^a$ with respect to $A$, such that
	\begin{equation*}
		A(p^{it}) = \{1, p^t, p^{2t},\ldots , p^{it}\}
	\end{equation*}
	for every $i\in \{0, 1,\ldots, a/t\}$.
	
	Given a regular system $A$, an integer $n>1$ is called primitive with respect to $A$ ($A$-primitive) 
	if $A(n)=\{1,n\}$. By (i) if $n$ is primitive, then $n=p^a$ for some prime $p$ and $a\ge 1$. Furthermore, 
	a prime power $p^a$ ($a\ge 1$) is primitive if and only if $t_A(p^a)=a$. It turns out that
	\begin{equation*}
		\mu_A(p^a)=\begin{cases} -1, & \text{ if $p^a$ is primitive},\\ 0, & \text{ otherwise}. 
		\end{cases}    
	\end{equation*}
	
	Examples of regular systems of divisors are $A = D$, where $D(n)$ is the set
	of all positive divisors of $n$ and $A = U$, where $U(n)$ is the set of unitary divisors of $n$.
	For every prime power $p^a$ ($a\ge 1$) one has $t_D(p^a) = 1$, $D(p^a)=\{1,p,p^2,\ldots,p^a\}$, the primitive integers with respect to $D$ are the primes, and $t_U(p^a) = a$, $U(p^a)=\{1,p^a\}$, the primitive integers with respect to $U$ being the prime powers $p^a$ ($a\ge 1$). Here $*_D$ and $*_U$ are the Dirichlet convolution and the unitary convolution, respectively. 
	
	To give another example of a regular system of divisors let, for every prime $p$ and $a\in \N$, $t_A(p^a)=2$ if 
	$a$ is even, and $t_A(p^a)=a$ if $a$ is odd.
	That is, $A(p^a)=\{1,p^2,p^4,\ldots,p^a\}$ if $a$ is even, and $A(p^a)=\{1,p^a\}$ if $a$ is odd.
	
	Let $A=(A(n))_{n\in\N}$ be a regular system of divisors and $k\in \N$. Define $A_k=(A_k(n))_{n\in \N}$, where 
	$A_k(n)=\{d\in \N: d^k\in A(n^k)\}$. Then the system $A_k$ is also regular, see Sita Ramaiah \cite[Th.\ 3.1]{Sit1978}.
	Let $(j,n)_{A,k}= \max \{d^k: d^k\mid j, d^k\in A(n)\}$. Then  
	\begin{equation} \label{Phi_k_gen}
		\Phi_{A,k,S,n}(x) = \prod_{\substack{1\le j\le n^k \\ ((j,n^k)_{A,k})^{1/k} \in S}} \left(x- e^{2\pi ij/n^k} \right)
	\end{equation}
	is the polynomial corresponding to the generalized Ramanujan sums, investigated in the literature. See Haukkanen 
	\cite[Sect.\ 5]{Hau1988}. We will not consider \eqref{Phi_k_gen} in what follows.
	
	For properties of regular convolutions and related arithmetical functions we refer to 
	Narkiewicz \cite{Nar1963}, McCarthy \cite{McC1968,McC1986}, Sita Ramaiah \cite{Sit1978}.
	
	\subsection{Generalized arithmetic functions} \label{Sect_func}
	
	If $A$ is a regular system of divisors, then the corresponding generalized Euler function $\varphi_A(n)$ 
	is defined as  
	\begin{equation*} 
		\varphi_A(n) = \sum_{\substack{j=1 \\ (j,n)_A=1}}^n 1,
	\end{equation*}
	where $(j,n)_A$ is  given by \eqref{(j,n)_A}. 
	In the proofs we will use the property that 
	\begin{align} \label{prop_gcd_A}
		d\in A((j,n)_A) \text{ holds if and only if } d\mid j \text{ and } d\in A(n).
	\end{align}
	
	We note that 
	\begin{equation} \label{varphi_A}
		\varphi_A(n) = \sum_{d\in A(n)} d \mu_A(n/d) = n \prod_{p^a\mid \mid n} \left(1-\frac1{p^t}\right),
	\end{equation}
	where $t=t_A(p^a)$ is the type of $p^a$. Hence $\varphi_A$ is multiplicative. 
	Here $\varphi_D(n)=\varphi(n)$ is the classical Euler function, and 
	$\varphi_U(n)=\varphi^*(n)$ is its unitary analogue. 
	
	Now let $S\subseteq \N$ be an arbitrary (nonempty) set and let $\varrho_S$ be the characteristic function of $S$. Given a regular system of divisors $A$, define the M\"obius-type function $\mu_{A,S}$ by 
	\begin{equation} \label{varrho_S}
		\sum_{d\in A(n)} \mu_{A,S}(d)= \varrho_S(n) \quad (n\in \N),
	\end{equation}
	that is,
	\begin{equation} \label{mu_A_S}
		\mu_{A,S}(n)= \sum_{d\in A(n)} \mu_A(d) \varrho_S(n/d).
	\end{equation}
	
	We say that $S$ is multiplicative if its characteristic function $\varrho_S$ is not identical zero and 
	multiplicative (hence $1\in S$). In this case $\mu_{A,S}$ is multiplicative and $\mu_{A,S}(p^a)=\varrho(p^a)-\varrho(p^{a-1}) 
	\in \{-1,0,1\}$ for every prime power $p^a$ ($a\ge 1$). If $S$ is not multiplicative, then $\mu_{A,S}$ can be unbounded.
	For example, if $S=\P$ is the set of primes and $A=D$, then $\mu_{D,\P}(p_1\cdots p_k)=(-1)^{k-1}k$ for distinct primes
	$p_1,\ldots,p_k$.
	
	The Euler-type function $\varphi_{A,S}(n)$ is defined by 
	\begin{equation*}
		\varphi_{A,S}(n) =  \sum_{\substack{j=1 \\ (j,n)_A\in S}}^n 1,
	\end{equation*}
	which reduces to $\varphi_A(n)$ if $S=\{1\}$. More generally, the corresponding Ramanujan sums 
	$c_{A,S,n}(k)$ are given by
	\begin{equation*}
		c_{A,S,n}(k) = \sum_{\substack{j=1\\ (j,n)_A\in S}}^n \zeta_n^{jk},
	\end{equation*}
	and they share properties similar to the classical Ramanujan sums $c_n(k)$. For example, for every regular $A$, every subset $S$ and $k,n\in \N$,
	\begin{equation} \label{c_A_S_id}
		c_{A,S,n}(k)= \sum_{d\mid (k,n)_A} d\mu_{A,S}(n/d)
	\end{equation}
	holds, therefore all values of $c_{A,S,n}(k)$ are real integers. In particular, $c_{A,S,n}(1) =\mu_{A,S}(n)$, and
	\begin{equation} \label{form_Euler_A_S}
		c_{A,S,n}(0)  =\varphi_{A,S}(n) = \sum_{d\in A(n)} d \mu_{A,S}(n/d).
	\end{equation}
	
	If $S$ is multiplicative, then $c_{A,S,n}(k)$ and $\varphi_{A,S}(n)$ are multiplicative in $n$.
	If $S=\{1\}$, then one has the H\"older-type identity
	\begin{equation} \label{Holder}
		c_{A,n}(k) := c_{A,\{1\},n}(k)= \frac{\varphi_A(n)\mu_A(n/(k,n)_A)}{\varphi_A(n/(k,n)_A)},
	\end{equation}
	see McCarthy \cite[p.\ 170]{McC1986}.
	
	In the case $A=D$ these functions were introduced by Cohen \cite{Coh1959}. Also see the author \cite{Tot2004}. For arbitrary $A$ and $S$ see Haukkanen \cite{Hau1988}, the author and Haukkanen \cite{TotHau1996, TotHau2000}.
	
	\section{Results} \label{Section_Results}
	
	\subsection{Basic properties}
	
	Consider the cyclotomic polynomials $\Phi_{A,S,n}(x)$ defined by \eqref{Phi_A_S}.
	It follows from the definition that these are monic polynomials, and the degree of $\Phi_{A,S,n}(x)$ is 
	$\varphi_{A,S}(n)$.
	
	\begin{theorem} \label{Th_A_S} For every regular system $A$, every subset $S$
		and $n\in \N$ we have 
		\begin{align}
			\Phi_{A,S,n}(x) & = \prod_{d\in A(n)} \left(x^d-1 \right)^{\mu_{A,S}(n/d)}  \label{Phi_gen_1} \\ 
			& = (-1)^{\varrho_S(n)} \prod_{d\in A(n)} \left(1-x^d\right)^{\mu_{A,S}(n/d)}   \label{Phi_gen_1_bis} \\ 
			& = \prod_{\substack{d\in A(n)\\ n/d\in S}} \Phi_{A,d}(x). \label{Phi_gen_2}  
		\end{align}
		
		Furthermore,
		\begin{align}                
			\prod_{d\in A(n)} \Phi_{A,S,d}(x) & = \prod_{\substack{d\in A(n)\\n/d\in S}} (x^d-1), \label{Phi_gen_3} \\
			\prod_{d\in A(n)} \Phi_{A,d}(x) & = x^n-1. \label{Phi_gen_4}
		\end{align}
	\end{theorem}
	
	If $A=D$ or $A=U$, respectively $S=\{1\}$, $S=\{m^2: m\in \N\}$ or $S=\N \setminus \{1\}$, then we recover properties of the classical cyclotomic polynomials and its analogues mentioned in the Introduction.
	
	\begin{corollary} \label{Cor_pal} For every $A$, $S$ and $n\in \N$,
		\begin{equation*}
			\Phi_{A,S,n}(x)= (-1)^{\varrho_S(n)} x^{\varphi_{A,S}(n)} \Phi_{A,S,n}(1/x),   
		\end{equation*}
		hence the polynomial $\Phi_{A,S,n}(x)$ is palindromic or antipalindromic, according to $n\in S$ or $n\notin S$.
	\end{corollary}
	
	One may wonder what will be the generalized identity corresponding to \eqref{x_n_unit} and \eqref{x_n_square}. 
	The answer is included in the next corollary.
	
	\begin{corollary} \label{Cor_x_n_1} Assume that $1\in S$. Then we have 
		\begin{equation} \label{id_x_n_1}
			x^n-1 = \prod_{d\in A(n)} \Phi_{A,S,d}(x)^{h_{A,S}(n/d)},
		\end{equation}
		where $h_{A,S}$ is the inverse with respect to $A$-convolution of the function $\mu_{A,S}$.
	\end{corollary}
	
	Here the values of $h_{A,S}$ 
	can be computed for special choices of $A$ and $S$. To give another example, let $A=U$ and $S=\{m^2: m\in \N\}$. Then it
	turns out that $h_{U,S}$ is the characteristic function of the exponentially odd integers, i.e., 
	integers with all exponents odd in their prime power factorization. Therefore, considering the unitary analogue of 
	$Q_n(x)$, given by \eqref{def_Q_n}, namely
	\begin{equation*} 
		Q_n^*(x) = \prod_{\substack{j=1\\ (j,n)_*\text { a square}}}^n \left(x- \zeta_n^j \right),
	\end{equation*}
	we have 
	\begin{equation*}
		x^n-1 = \prod_{\substack{d\mid \mid n\\ n/d \text{ exponentially odd}}} Q_d^*(x).
	\end{equation*}
	
	For a regular system $A=(A(n))_{n\in \N}$ of divisors define the $A$-kernel function $\kappa_A$ by 
	$\kappa_A(n) = \prod_{p^a \mid \mid n}p^t$, where $t = t_A(p^a)$ is the type of $p^a$. Here $\kappa_D(n)=\kappa(n)=\prod_{p\mid n} p$ is the square-free kernel of $n$, and $\kappa_U(n)=n$ ($n\in \N$).
	
	\begin{corollary} \label{Cor_kappa} For every regular $A$ and $n\in \N$,
		\begin{equation*}
			\Phi_{A,n}(x)= \Phi_{\kappa_A(n)}(x^{n/\kappa_A(n)}).   
		\end{equation*}
	\end{corollary}
	
	For a regular system
	$A=(A(n))_{n\in \N}$ of divisors the $A$-core function $\gamma_A$ is defined by 
	$\gamma_A(n) = n\kappa(n)/\kappa_A(n)=\prod_{p^a \mid \mid n}p^{a-t+1}$, where $t = t_A(p^a)$ is the type of $p^a$. See McCarthy \cite[p.\ 166]{McC1986}. 
	Note that $\gamma_D(n)=n$ ($n\in \N$) and $\gamma_U(n)=\kappa(n)=\prod_{p\mid n} p$.
	
	The following result is a generalization of identity \eqref{Phi_star_prod}.
	
	\begin{theorem} \label{Th_A} For every regular system $A$ of divisors and every $n\in \N$,
		\begin{equation}  \label{Phi_A_irred}
			\Phi_{A,n}(x)= \prod_{\substack{d\mid n\\ \gamma_A(n)\mid d}} \Phi_d(x).
		\end{equation}
	\end{theorem}
	
	As a corollary we deduce the following more general identity.
	
	\begin{corollary} \label{Cor_gen_irred} For every regular $A$, every subset $S$ and $n\in \N$ we have
		\begin{equation*} 
			\Phi_{A,S,n}(x)= \prod_{\substack{d\in A(n)\\ n/d\in S}} \prod_{\substack{e\mid d\\ \gamma_A(d)\mid e}} \Phi_e(x),
		\end{equation*}
		therefore all the polynomials $\Phi_{A,S,n}(x)$ have integer coefficients.
	\end{corollary}
	
	Theorem \ref{Th_A} is a special case of the following general result.
	
	\begin{theorem} \label{Th_general} Let $A$ be a regular system of divisors,
		and let the functions $g$ and $g_A$ be defined for every $n\in \N$ by
		\begin{equation*} 
			\sum_{d\mid n} g(d) = \sum_{d\in A(n)} g_A(d).
		\end{equation*}
		Then
		\begin{equation*} 
			g_A(n)= \sum_{\substack{d\mid n\\ \gamma_A(n)\mid d}} g(d).
		\end{equation*}
	\end{theorem}
	
	Another application of Theorem \ref{Th_general} is to the Ramanujan sums $c_{A,n}(k) = c_{A,\{1\},n}(k)$, namely
	\begin{equation} \label{id_Ramanujan_gamma} 
		c_{A,n}(k)= \sum_{\substack{d\mid n\\ \gamma_A(n)\mid d}} c_d(k),
	\end{equation}
	in particular,
	\begin{align*} 
		\varphi_A(n) & = \sum_{\substack{d\mid n\\ \gamma_A(n)\mid d}} \varphi(d), \\
		\mu_A(n)     & = \sum_{\substack{d\mid n\\ \gamma_A(n)\mid d}} \mu(d).
	\end{align*}
	
	The last three identities are given and proved by McCarthy \cite[Th.\ 2, Cor.\ 2.1, Cor.\ 2.2]
	{McC1968} by different arguments, namely by using properties of finite Fourier representations of $n$-even functions. Also see McCarthy \cite[pp.\ 165--167]{McC1986}. Our proof is direct and short. See Cohen \cite[Lemma\ 3.1]{Coh1961} for a different approach in the case $A=U$.
	
	\subsection{Other identities}
	
	\begin{theorem} \label{Th_repr_cos} Let $A$ be a regular system of divisors and let $n\in \N$. Then for $x>1$ \textup{(or formally)},
		\begin{align} 
			\Phi_{A,n}(x) = \prod_{j=1}^n \left(x^{(j,n)_A}-1\right)^{\cos (2\pi j/n)} \label{id_cos}.
		\end{align}
	\end{theorem}
	
	For $A=D$ identity \eqref{id_cos} was proved by Schramm \cite{Sch2015} and for $A=U$
	by Moree and the author \cite[Th.\ 5]{MorTot2020}.
	In fact, \eqref{id_cos} is a special case of the following general result and its corollary concerning the discrete Fourier
	transform (DFT) of functions involving the quantity $(j,n)_A$.
	
	\begin{theorem} \label{Th_zeta_general} 
		Let $A$ be a regular system of divisors, $f:\N \to \C$ be an arbitrary arithmetic function and $n,k\in \N$. Then
		\begin{equation} \label{id_f_exp}
			\sum_{j=1}^n f((j,n)_A)\zeta_n^{jk} = \sum_{d\in (k,n)_A} d\, (\mu_A*_A f)(n/d).
		\end{equation}
	\end{theorem}
	
	\begin{corollary} \label{Cor_zeta_f_real} 
		If $f$ is a real valued function, then 
		\begin{align} \label{id_f_cos}
			\sum_{j=1}^n f((j,n)_A)\cos (2\pi jk/n)  & = \sum_{d\in (k,n)_A} d\, (\mu_A*_A f)(n/d),\\
			\nonumber
			\sum_{j=1}^n f((j,n)_A)\sin (2\pi jk/n)  & = 0.
		\end{align}
	\end{corollary}
	
	Now we present a Menon-type identity, involving $(j-1,n)_A$, where $j$ runs over an $A$-reduced residue system (mod $n$), 
	that is, $(j,n)_A=1$. 
	
	\begin{theorem} \label{Th_Menon_type} Let $A$ be a regular system of divisors and let $n\in \N$.
		Then 
		\begin{align} \label{prod_Menon}
			\prod_{\substack{j=1\\ (j,n)_A=1}}^n (x^{(j-1,n)_A}-1) & = \prod_{d\in A(n)} \Phi_{A,d}(x)^{\varphi_A(n)/\varphi_A(d)}, 
		\end{align}
	\end{theorem}
	
	If $A=D$, then \eqref{prod_Menon} was given by the author \cite{Tot2023problem}, \cite[Eq.\ (45)]{Tot2023}.
	Here \eqref{prod_Menon} can be deduced from the following general result, known in the literature, even 
	in a more general form, see \cite[Th.\ 9.1]{Sit1978}. However, for the sake of completeness we also 
	give a direct and short proof of it. 
	
	\begin{theorem} \label{Th_Menon_gen} Let $A$ be a regular system of divisors and $f:\N \to \C$ an arbitrary arithmetic function. Then for every $n\in \N$,
		\begin{equation} \label{id_f}
			\sum_{\substack{j=1\\ (j,n)_A=1}}^n f((j-1,n)_A) = \varphi_A(n) \sum_{d\in A(n)} \frac{(\mu_A*_A f)(d)}{\varphi_A(d)}.
		\end{equation}
	\end{theorem}
	
	Note that if $A=D$ and $f(n)=n$ ($n\in \N$), then identity \eqref{id_f} recovers
	\eqref{Menon_id_origi}.
	
	Another Menon-type identity is the following.
	
	\begin{theorem} \label{Th_Menon_type_2} Let $A$ be a regular system of divisors, let $n\in \N$
		and $\chi$ be a Dirichlet character \textup{(mod $n$)} with conductor $d$ \textup{($d\mid n$)}. 
		Then for real $x>1$ \textup{(or formally)},
		\begin{equation} \label{Menon_A_new}
			\prod_{j=1}^n \left(x^{(j-1,n)_A}-1\right)^{\Real(\chi(j))} = \prod_{d \delta \in A(n)} 
			\Phi_{A, d\delta}(x)^{\varphi(n)/\varphi(d\delta)}. 
		\end{equation}
		
		If $\chi$ is a primitive character \textup{(mod $n$)}, then 
		\begin{equation} \label{Menon_A_new_primitive}
			\Phi_{A,n}(x)= \prod_{j=1}^n \left(x^{(j-1,n)_A}-1\right)^{\Real(\chi(j))}.
		\end{equation}
	\end{theorem}
	
	If $A=D$, then \eqref{Menon_A_new} and \eqref{Menon_A_new_primitive} recover identities
	\eqref{Menon_D_new} and \eqref{Menon_D_new_primitive}, respectively. See the author \cite{Tot2018}
	for some related identities in the case $A=D$.
	
	Theorem \ref{Th_Menon_type_2} can be deduced from the next general result and its corollary.
	
	\begin{theorem} \label{Th_chi_gen} Let $A$ be a regular system of divisors, $f:\N \to \C$ be an arbitrary arithmetic function, 
		let $n\in \N$ and $\chi$ be a Dirichlet character \textup{(mod $n$)} with conductor $d$ \textup{($d\mid n$)}.
		Then 
		\begin{equation} \label{Menon_A_new_f}
			\sum_{j=1}^n f((j-1,n)_A) \chi(j) = \varphi(n) \sum_{d \delta \in A(n)} 
			\frac{(\mu_A *_A f)(d\delta)}{\varphi(d\delta)}. 
		\end{equation}
		
		If $\chi$ is a primitive character \textup{(mod $n$)}, then 
		\begin{equation} \label{Menon_A_new_primit_f}
			\sum_{j=1}^n f((j-1,n)_A) \chi(j) = (\mu_A *_A f)(n). 
		\end{equation}
	\end{theorem}
	
	\begin{corollary}  \label{Cor_f_real}
		Let $f$ be a real valued function. If $\chi$ is a Dirichlet character \textup{(mod $n$)} with conductor 
		$d$ \textup{($d\mid n$)}, then 
		\begin{align*} 
			\sum_{j=1}^n f((j-1,n)_A) \Real(\chi(j)) & = \varphi(n) \sum_{d \delta \in A(n)} 
			\frac{(\mu_A *_A f)(d\delta)}{\varphi(d\delta)}, \\
			\sum_{j=1}^n f((j-1,n)_A) \Imag(\chi(j)) & = 0.
		\end{align*}
		
		If $\chi$ is a primitive character \textup{(mod $n$)}, then 
		\begin{align*} 
			\sum_{j=1}^n f((j-1,n)_A) \Real(\chi(j)) & = (\mu_A *_A f)(n).
		\end{align*}
	\end{corollary}
	
	It is possible to deduce some further related identities. For example, we have the next result, known
	in the case $A=D$, and proved by Moree and the author \cite[Cor.\ 6]{MorTot2020} for $A=U$.
	
	\begin{theorem} \label{Th_coonect_Ramanujan_sum} For every regular system $A$, $n>1$ and $x\in \C$, 
		$|x|<1$ \textup{(or formally)},
		\begin{equation} \label{Phi_c}
			\Phi_{A,n}(x) = \exp \left(-\sum_{k=1}^{\infty} \frac{c_{A,n}(k)}{k}x^k \right).
		\end{equation}
	\end{theorem}
	
	\subsection{Coefficients} 
	
	Now consider the coefficients of the monic polynomials $\Phi_{A,S,n}(x)$ of degree $\varphi_{A,S}(n)$.  
	It follows from identity \eqref{c_A_S_id} applied for $k=1$ that the
	coefficient of the term $x^{\varphi_{A,S}(n)-1}$ is $-c_{A,S,n}(1)= -\mu_{S,A}(n)$. 
	In order to deduce formulas for the other coefficients as well, let
	\begin{equation*}
		\Phi_{A,S,n}(x)= \sum_{j=0}^{\varphi_{A,S}(n)} a_{A,S,n}(j) x^j.
	\end{equation*}
	
	We have the following generalization of the M\"oller-Endo identities. 
	
	\begin{theorem} \label{Th_binom} For every $A,S,n,k$,
		\begin{equation} \label{form_coeff} 
			a_{A,S,n}(k)= (-1)^{\varrho_S(n)} \sum_{\substack{j_1,j_2,\ldots,j_k\ge 0\\ j_1+2j_2+\cdots +kj_k=k}}
			\prod_{d=1}^k (-1)^{j_d} \binom{\mu_{A,S}(n/d)}{j_d},
		\end{equation}
		with the convention $\mu_{A,S}(t)=0$ if $t$ is not an integer.
	\end{theorem}
	
	This shows that $a_{A,S,n}(1)= - (-1)^{\varrho_S(n)} \mu_{A,S}(n)$. Hence, according to Corollary \ref{Cor_pal}, 
	\begin{align*}
		a_{A,S,n}(\varphi_{A,S}(n)-1)= - \mu_{A,S}(n),
	\end{align*}
	as mentioned above. Also,
	\begin{align*}
		(-1)^{\varrho_S(n)} a_{A,S,n}(2)= a_{A,S,n}(\varphi_{A,S}(n)-2)= \frac{\mu_{A,S}(n)(\mu_{A,S}(n)-1)}{2} -\mu_{A,S}(n/2),
	\end{align*}
	and so on, similar to the classical case.
	
	Now let $S=\{1\}$ and let $a_{A,n}(k) := a_{A,\{1\},n}(k)$ denote the coefficients of
	$\Phi_{A,n}(x)$. The identity of Corollary \ref{Cor_kappa} shows that to study these coefficients
	it is enough to consider the case when $n$ is replaced by $\kappa_A(n)$.
	As a generalization of the Grytczuk-Tropak recursion formula we prove the following result.
	
	\begin{theorem} \label{Th_recursion}
		Let $A$ be a regular system of divisors. If $n$ is a product of $A$-primitive integers, then 
		for every $k$ with $1\le k \le \varphi_{A}(n)$,
		\begin{equation} \label{recursion}
			a_{A,n}(k)= -\frac{\mu_A(n)}{k} \sum_{j=1}^k a_{A,n}(k-j) \mu_A((j,n)_A)\varphi_A((j,n)_A),
		\end{equation}
		where $a_{A,n}(0)=1$.
	\end{theorem}
	
	Note that in the classical case ($A=D$) \eqref{recursion} only holds for squarefree values of $n$ (fact 
	omitted in some texts). However, in the unitary case ($A=U$) \eqref{recursion} holds for every $n\in \N$.
	
	\section{Proofs} \label{Section_Proofs}
	
	\begin{proof}[Proof of Theorem {\rm \ref{Th_A_S}}] More generally, let $f:\N\to \C$ be an arbitrary function and $F_f(n):=\sum_{j=1}^n f(j/n)$. Then 
		\begin{equation*}
			S_{f,A,S}(n):= \sum_{\substack{j=1\\ (j,n)_A\in S}}^n f(j/n)= \sum_{j=1}^n f(j/n) \varrho_S((j,n)_A)=
			\sum_{j=1}^n f(j/n) \sum_{d\in A((j,n)_A)} \mu_{A,S}(d),
		\end{equation*}
		by \eqref{varrho_S}. Using property \eqref{prop_gcd_A} we deduce that 
		\begin{equation*}
			S_{f,A,S}(n)= \sum_{d\in A(n)} \mu_{A,S}(d) \sum_{\substack{j=1\\ d\mid j}}^n f(j/n) = 
			\sum_{d\in A(n)} \mu_{A,S}(d) F_f(n/d). 
		\end{equation*}
		
		Note that this is generalization of the the Hurwitz lemma, recovered for $A=D$, $S=\{1\}$. 
		Now if (formally) $f(n)=\log (x-e^{2\pi in})$, then $F_f(n)=\sum_{j=1}^n \log (x-e^{2\pi ij/n}) = \log (x^n-1)$, 
		and deduce that  
		\begin{equation} \label{sum_id}
			\sum_{\substack{j=1\\ (j,n)_A\in S}}^n \log (x-e^{2\pi ij/n}) = \sum_{d\in A(n)} \log(x^d-1) \mu_{A,S}(n/d),    
		\end{equation} 
		equivalent to \eqref{Phi_gen_1}. Now \eqref{Phi_gen_1_bis} follows by \eqref{Phi_gen_1} and \eqref{varrho_S}
		
		In terms of the $A$-convolution \eqref{sum_id} shows that
		\begin{equation} \label{id_first} 
			\log \Phi_{A,S,\DOT}(x) =  \log(x^{\DOT}-1) *_A \mu_{A,S},    
		\end{equation} 
		that is, using \eqref{mu_A_S},
		\begin{equation}  \label{log_Phi_A_S}
			\log \Phi_{A,S,\DOT}(x) =  \log(x^{\DOT}-1) *_A \mu_A *_A \varrho_S.    
		\end{equation} 
		
		If $S=\{1\}$, then \eqref{log_Phi_A_S}  gives
		\begin{equation}  \label{log_Phi_A}
			\log \Phi_{A,\DOT}(x) =  \log(x^{\DOT}-1) *_A \mu_A,   
		\end{equation} 
		and combining \eqref{log_Phi_A_S} and \eqref{log_Phi_A} we have
		\begin{equation*}
			\log \Phi_{A,S,\DOT}(x) = \log \Phi_{A,\DOT}(x) *_A \varrho_S,    
		\end{equation*} 
		giving \eqref{Phi_gen_2}.
		
		From \eqref{log_Phi_A_S} we also have 
		\begin{equation*}  
			\log \Phi_{A,S,\DOT}(x) *_A \1 =  \log(x^{\DOT}-1) *_A \varrho_S,    
		\end{equation*} 
		where $\1(n)=1$ ($n\in \N$). This shows the validity of \eqref{Phi_gen_3},
		which reduces to \eqref{Phi_gen_4} if $S=\{1\}$.
	\end{proof}
	
	\begin{proof}[Proof of Corollary {\rm \ref{Cor_pal}}]
		This is a direct consequence of identities \eqref{Phi_gen_1}, \eqref{Phi_gen_1_bis} and \eqref{form_Euler_A_S}.
	\end{proof}
	
	\begin{proof}[Proof of Corollary {\rm \ref{Cor_x_n_1}}]
		If $1\in S$, then $\mu_{A,S}(1)=\mu_A(1)\varrho(1)=1\ne 0$. Hence the function $\mu_{A,S}$ has an inverse with respect 
		to $A$-convolution, we denote it by $h_{A,S}$. That is, $h_{A,S} *_A \mu_{A,S}=\varepsilon$, where $\varepsilon(n)=\lfloor 1/n \rfloor$ ($n\in \N$). From
		\eqref{id_first} we obtain that
		\begin{equation*}  
			\log(x^{\DOT}-1) = \log \Phi_{A,S,\DOT}(x) *_A h_{A,S},    
		\end{equation*} 
		equivalent to \eqref{id_x_n_1}.
	\end{proof}
	
	\begin{proof}[Proof of Corollary {\rm \ref{Cor_kappa}}] 
		By the definition of the function $\mu_A$,
		if $d$ is not a product of $A$-primitive integers, then $\mu_A(d)=0$.
		Therefore, for every $n\in \N$, using \eqref{Phi_gen_1},
		\begin{align} \nonumber
			\Phi_{A,n}(x) & = \prod_{d\in A(n)} \left(x^{n/d}-1 \right)^{\mu_{A}(d)}  
			= \prod_{\substack{d\in A(n)\\ d\in A(\kappa_A(n))}} \left(x^{n/d}-1 \right)^{\mu_{A}(d)} \\
			&  = \prod_{d\in A(\kappa_A(n))} \left((x^{n/\kappa(n)})^{\kappa(n)/d}-1 \right)^{\mu_{A}(d)} 
			= \Phi_{A,\kappa(n)}(x^{n/\kappa(n)}). \nonumber
		\end{align}
	\end{proof}

	\begin{proof}[Proof of Theorem {\rm \ref{Th_A}}]
		Apply (formally) Theorem \ref{Th_general} in the case $g(n)=\log \Phi_n(x)$, $g_A(n)=\log \Phi_{A,n}(x)$, where 
		$f(n)=\log (x^n-1)$ by taking into account \eqref{Phi_gen_4}. We deduce that
		\begin{equation*}
			\log \Phi_{A,n}(x)= \sum_{\substack{d\mid n\\ \gamma_A(n)\mid d}} \log \Phi_d(x),
		\end{equation*}
		which gives \eqref{Phi_A_irred}.
	\end{proof}
	
	\begin{proof}[Proof of Corollary {\rm \ref{Cor_gen_irred}}] This is a direct consequence of identities
		\eqref{Phi_gen_2} and \eqref{Phi_A_irred}.
	\end{proof}
	
	\begin{proof}[Proof of Theorem {\rm \ref{Th_general}}]
		Let 
		\begin{equation*} 
			f(n)= \sum_{d\mid n} g(d) = \sum_{d\in A(n)} g_A(d) \quad (n\in \N).
		\end{equation*}
		
		Then we have
		\begin{equation*} 
			g_A(n)= \sum_{d\in A(n)} f(d)\mu_A(n/d) = \sum_{d\in A(n)} \mu_A(n/d) \sum_{\delta \mid d} g(\delta)
		\end{equation*}
		\begin{equation*} 
			= \sum_{\substack{\delta jm =n\\ \delta j\in A(n)}} g(\delta) \mu_A(m) 
			=\sum_{\delta t=n} g(\delta) \sum_{\substack{jm =t\\ \delta j\in A(n)}} \mu_A(m).
		\end{equation*}
		
		Since for every regular system $A$, $d\in A(n)$ holds if and only if $n/d\in A(n)$ we deduce that
		$\delta j\in A(n)$ if and only if $n/(\delta j)=m \in A$, and have
		\begin{equation*} 
			g_A(n) =\sum_{\delta t=n} g(\delta) \sum_{\substack{jm =t\\ m \in A(n)}} \mu_A(m)=
			\sum_{\delta \mid n} g(\delta) \sum_{\substack{m \mid n/\delta \\ m \in A(n)}} \mu_A(m).
		\end{equation*}
		
		We show that for every fixed $n$ and $\delta$ with $\delta \mid n$,
		\begin{equation*} 
			\sum_{\substack{m \mid n/\delta \\ m \in A(n)}} \mu_A(m) = \begin{cases} 1, & \text{ if } \gamma_A(n)\mid \delta, 
				\\ 0, &\text{ otherwise,} \end{cases}
		\end{equation*}
		which will finish the proof.
		
		To do this, let $n=p_1^{a_1}\cdots p_r^{a_r}$, $\delta=p_1^{b_1}\cdots p_r^{b_r}$ with
		$0\le b_i\le a_i$ ($1\le i\le r$). Note that $m\in A(n)$ holds if and only if
		$m=p_1^{c_1t_1}\cdots p_r^{c_rt_r}$ with $0\le c_i\le a_i/t_i$, where $t_i$ is the type of $p_i^{a_i}$ ($1\le i \le r$). Therefore,
		since the function $\mu_A$ is multiplicative,
		\begin{equation*}
			\sum_{\substack{m \mid n/\delta \\ m \in A(n)}} \mu_A(m) = \sum_{m=p_1^{c_1t_1}\cdots p_r^{c_rt_r} \, \mid \, p_1^{a_1-b_1}\cdots p_r^{a_r-b_r}} \mu_A(m)
		\end{equation*}
		\begin{equation*}
			= \sum_{0\le c_1\le (a_1-b_1)/t_1} \mu_A(p_1^{c_1t_1}) 
			\cdots \sum_{0\le c_r\le (a_r-b_r)/t_r} \mu_A(p_r^{c_rt_r}) 
		\end{equation*}
		\begin{equation*}
			= \sum_{0\le c_1\le (a_1-b_1)/t_1} \mu(p_1^{c_1}) 
			\cdots \sum_{0\le c_r\le (a_r-b_r)/t_r} \mu(p_r^{c_r}), 
		\end{equation*}
		where $\mu$ is the classical M\"{o}bius function.
		
		Here if $(a_i-b_i)/t_i\ge 1$ for some $i$, then $\sum_{c_i} \mu(p_i^{c_1})=\mu(1)+\mu(p)=0$, and 
		the product of the sums is also zero. Otherwise, $(a_i-b_i)/t_i< 1$ for all $i$ holds if and only if $a_i-b_i<t_i$ 
		for all $i$, equivalent to $a_i-t_i+1\le b_i$ for all $i$, that is, $\gamma_A(n)\mid \delta$. In this case for all $i$,
		$\sum_{c_i} \mu(p_i^{c_i})=\mu(1)=1$, and the proof is ready.
	\end{proof}
	
	\begin{proof}[Proof of Theorem {\rm \ref{Th_repr_cos}}]
		Apply identity \eqref{id_f_cos} to the function $f(n)=\log (x^n-1)$, where $x>1$ is real (or formally), and take 
		into account that $\mu_A *_A f= \Phi_{A,\DOT}(x)$ by \eqref{Phi_gen_4} and M\"{o}bius inversion.
	\end{proof}
	
	\begin{proof}[Proof of Theorem {\rm \ref{Th_zeta_general}}]
		We have by using that $f(n)=\sum_{d\in A(n)} (\mu_A *_A f)(d)$ ($n\in \N$) and property \eqref{prop_gcd_A},
		\begin{equation*}
			\sum_{j=1}^n f((j,n)_A)\zeta_n^{jk}= \sum_{j=1}^n \zeta_n^{jk} \sum_{d\in A((j,n)_A)}
			(\mu_A*_A f)(d)
			= \sum_{j=1}^n \zeta_n^{jk} \sum_{\substack{d \mid j\\ d\in A(n)}}
			(\mu_A*_A f)(d)
		\end{equation*}
		\begin{equation*}
			= \sum_{d\in A(n)} (\mu_A *_A f)(d) \sum_{\substack{j=1\\d\mid j}}^n \zeta_n^{jk}
			= \sum_{\substack{d\in A(n)\\ n/d \mid k}} (\mu_A *_A f)(d) \frac{n}{d} 
			= \sum_{\substack{d\in A(n)\\ d\mid k}} d (\mu_A *_A f)(n/d),  
		\end{equation*}
		giving \eqref{id_f_exp} by using \eqref{prop_gcd_A} again.
	\end{proof}
	
	\begin{proof}[Proof of Corollary {\rm \ref{Cor_zeta_f_real}}]
		If $f$ is a real valued function, then the right hand side of \eqref{id_f_exp}
		is real, and we deduced the given identities.
	\end{proof}
	
	\begin{proof}[Proof of Theorem {\rm \ref{Th_Menon_type}}]
		Apply Theorem \ref{Th_Menon_gen} by taking (formally) $f(n)=\log (x^n-1)$.
		Then  $\mu_A *_A f= \log \Phi_{A,\DOT}(x)$ by \eqref{Phi_gen_4} and M\"{o}bius inversion.
		We deduce that  
		\begin{equation*}
			\sum_{\substack{j=1\\ (j,n)_A=1}}^n \log \left(x^{(j-1,n)_A}-1\right) = \varphi_A(n) \sum_{d\in A(n)} \frac{\log \Phi_{A,d}(x)}{\varphi_A(d)},
		\end{equation*}
		equivalent to the given identity.
	\end{proof}
	
	\begin{proof}[Proof of Theorem {\rm \ref{Th_Menon_gen}}]
		Let $M_{A,f}(n)$ denote the left hand side of \eqref{id_f}. We have by property \eqref{prop_gcd_A},
		\begin{equation*} 
			M_{A,f}(n) =  \sum_{j=1}^n f((j-1,n)_A) \sum_{d\mid (j,n)_A} \mu_A(d) 
			= \sum_{d\in A(n)} \mu_A(d) \sum_{\substack{j=1\\ d\mid j}}^n f((j-1,n)_A).
		\end{equation*}
		
		By using that $f(n)=\sum_{d\in A(n)} (\mu_A *_A f)(d)$ ($n\in \N$), we deduce 
		\begin{equation*}
			S_{A,f,d}(n):= \sum_{\substack{j=1\\ d\mid j}}^n f((j-1,n)_A) = \sum_{k=1}^{n/d} f((kd-1,n)_A) = \sum_{k=1}^{n/d} \sum_{e\in A((kd-1,n)_A)} (\mu_A*_A f)(e)
		\end{equation*}
		\begin{equation*}
			= \sum_{k=1}^{n/d} \sum_{\substack{e\mid kd-1\\ e\in A(n)}} (\mu_A*_A f)(e)= 
			\sum_{e\in A(n)} (\mu_A *_A f)(e) \sum_{\substack{k=1\\ kd\equiv 1\, \text{(mod $e$)}}}^{n/d} 1,
		\end{equation*}
		where the inner sum is $n/(de)$ if $(d,e)=1$ and $0$ otherwise. This gives
		\begin{equation*}
			S_{A,f,d}(n) = \frac{n}{d} \sum_{\substack{e\in A(n) \\(e,d)=1}} \frac{(\mu_A*_Af)(e)}{e}.
		\end{equation*}
		
		Thus
		\begin{equation*} 
			M_{A,f}(n)= \sum_{d\in A(n)} \mu_A(d) S_{A,f,d}(n) = 
			n \sum_{e\in A(n)} \frac{(\mu_A*_A f)(e)}{e} \sum_{\substack{d\in A(n) \\(d,e)=1}} \frac{\mu_A(d)}{d},
		\end{equation*}
		and for every $e\in A(n)$,
		\begin{equation*}
			\sum_{\substack{d\in A(n) \\(d,e)=1}} \frac{\mu_A(d)}{d}  =
			\prod_{\substack{p^a\mid \mid  n\\ p\nmid e}} \left(1-\frac1{p^t}\right)
		\end{equation*}
		\begin{equation*}
			= \prod_{p^a\mid \mid  n} \left(1-\frac1{p^t}\right) \prod_{\substack{p^a\mid \mid n\\ p\mid e}} \left(1-\frac1{p}\right)^{-1}
			= \frac{\varphi_A(n)}{n}\cdot \frac{e}{\varphi_A(e)},
		\end{equation*}
		where $t=t_A(p^a)$, by using \eqref{varphi_A} and the fact that if $e\in A(n)$ and $e=\prod p^b$, $n=\prod p^a$, then
		$t_A(p^b)=t_A(p^a)$ for all prime powers in question, see \cite[Cor.\ 4.2]{McC1986}.
		
		We obtain that 
		\begin{equation*}
			M_{A,f}(n)= \varphi_A(n) \sum_{e\in A(n)} \frac{(\mu_A*_Af)(e)}{\varphi_A(e)},
		\end{equation*}
		which is identity \eqref{id_f}.
	\end{proof}
	
	\begin{proof}[Proof of Theorem {\rm \ref{Th_Menon_type_2}}] 
		Apply Theorem \ref{Th_chi_gen} to the function $f(n)=\log (x^n-1)$, where $\mu_A*_A f = \log \Phi_{A,\DOT}(x)$.
	\end{proof}
	
	\begin{proof}[Proof of Theorem {\rm \ref{Th_chi_gen}}] 
		We need the following known results, see, e.g., \cite[Ch.\ 9]{MonVau2007}. 
		
		If $\chi$ is a  Dirichlet character (mod $n$) with conductor $d$, then there is a unique primitive character
		$\chi^*$ (mod $d$) that induces $\chi$. That is,
		\begin{equation} \label{induced_char}
			\chi(k) =  \begin{cases} \chi^*(k), & \text{ if $(k,n)=1$}, \\ 0, & \text{ if $(k,n)>1$}.
			\end{cases}
		\end{equation}
		
		Let $\chi$ be a primitive character (mod $n$). Then for every $d\mid n$, $d<n$ and every $s\in \Z$,
		\begin{equation} \label{chi_sum}
			\sum_{\substack{k=1\\ k\equiv s \, \text{\rm (mod $d$)} }}^n \chi(k)=0.
		\end{equation}
		
		We have, according to \eqref{induced_char},
		\begin{equation*}
			S_f:= \sum_{j=1}^n f((j-1,n)_A) \chi(j) = \sum_{\substack{j=1\\ (j,n)=1}}^n f((j-1,n)_A) \chi^*(j)
		\end{equation*}
		\begin{equation*}
			= \sum_{r=1}^d \sum_{\substack{j=1\\ (j,n)=1\\ j\equiv r \text{\rm (mod $d$)} }}^n f((j-1,n)_A) \chi^*(j)
			= \sum_{r=1}^d \chi^*(r) \sum_{\substack{j=1\\ (j,n)=1\\ j\equiv r \text{\rm (mod $d$)} }}^n f((j-1,n)_A).
		\end{equation*}
		
		Here, since $d\mid n$, if $(j,n)=1$ and $j\equiv r$ (mod $d$), then $(r,d)=(j,d)=1$. Therefore,
		the inner sum is empty in the case $(r,d)>1$.
		
		Now assume that $(r,d)=1$. We have
		\begin{equation*}
			T:= \sum_{\substack{j=1\\ (j,n)=1\\ j\equiv r \text{\rm (mod $d$)} }}^n f((j-1,n)_A) 
			= \sum_{\substack{j=1\\ (j,n)=1\\ j\equiv r \text{\rm (mod $d$)} }}^n \sum_{e \in A((j-1,n)_A)}
			(\mu_A *_A f)(e)
		\end{equation*}
		\begin{equation*}
			= \sum_{e\in A(n)} (\mu_A*_A f)(e) \sum_{\substack{j=1\\ (j,n)=1\\ j\equiv r \text{\rm (mod $d$)} \\ 
					j\equiv 1 \text{\rm (mod $e$)} }}^n 1, 
		\end{equation*}
		by property \eqref{prop_gcd_A}. Here the inner sum is 
		\begin{equation*}
			U:= \sum_{\substack{j=1\\ (j,n)=1\\ j\equiv r \text{ (mod $d$)} \\ 
					j\equiv 1 \text{ (mod $e$)} }}^n 1 = 
			\begin{cases} \frac{\varphi(n)(d,e)}{\varphi(de)}= \frac{\varphi(n)\varphi((d,e))}{\varphi(d)\varphi(e)} , & \text{ if $(d,e)\mid r-1$}, 
				\\ 0, & \text{ otherwise},
			\end{cases}
		\end{equation*}
		see the author \cite[Lemma 2.1]{Tot2019}. This gives 
		\begin{equation*}
			S_f=  \sum_{r=1}^d \chi^*(r) \frac{\varphi(n)}{\varphi(d)} \sum_{\substack{e \in A(n)\\ (d,e)\mid r-1 }} \frac{(\mu_A *_A f)(e)}{\varphi(e)} \varphi((d,e)) 
		\end{equation*}
		\begin{equation*}
			= \frac{\varphi(n)}{\varphi(d)} \sum_{e \in A(n)} \frac{(\mu_A *_A f)(e)}{\varphi(e)} \varphi((d,e)) \sum_{\substack{r=1\\ r\equiv 1 \text{ (mod $(d,e))$}}}^d \chi^*(r),
		\end{equation*}
		where by \eqref{chi_sum} the last sum is $0$ unless $(d,e)=d$, that is, $d\mid e$. We deduce 
		\begin{equation*}
			S_f= \frac{\varphi(n)}{\varphi(d)} \sum_{\substack{e \in A(n)\\ d\mid e}} \frac{(\mu_A *_A f)(e)}{\varphi(e)} \varphi(d)
			= \varphi(n)  \sum_{d\delta \in A(n)} \frac{(\mu_A *_A f)(d\delta)}{\varphi(d\delta)},
		\end{equation*}
		finishing the proof of \eqref{Menon_A_new_f}.
		
		If $\chi$ is a primitive character (mod $n$), then its conductor is $d=n$. Therefore, \eqref{Menon_A_new_f} reduces to \eqref{Menon_A_new_primit_f}.
	\end{proof}
	
	\begin{proof}[Proof of Corollary {\rm \ref{Cor_f_real}}]
		If $f$ is a real valued function, then the right hand side of \eqref{Menon_A_new_f}
		is real, and we deduced the given identities.
	\end{proof}
	
	\begin{proof}[Proof of Theorem {\rm \ref{Th_coonect_Ramanujan_sum}}]
		It is known that for every $n>1$ and $|x|<1$,
		\begin{equation} \label{cyclotomic_Ramanujan_exp}
			\Phi_n(x) = \exp \left(-\sum_{k=1}^{\infty} \frac{c_n(k)}{k}x^k \right),
		\end{equation}
		see, e.g., Herrera-Poyatos and Moree \cite[Eq.\ (1.2)]{HerMor2021}. By identities 
		\eqref{Phi_A_irred}, \eqref{cyclotomic_Ramanujan_exp} and \eqref{id_Ramanujan_gamma} we obtain 
		\begin{align*}
			\Phi_{A,n}(x)= \prod_{\substack{d\mid n\\ \gamma_A(n)\mid d}} \Phi_d(x) =\exp \left(-\sum_{k=1}^{\infty}\frac{x^k}{k}\sum_{\substack{d\mid n\\ \gamma_A(n)\mid d}} c_d(k) \right)= \exp \left(-\sum_{k=1}^{\infty} \frac{c_{A,n}(k)}{k}x^k 
			\right),
		\end{align*}
		completing the proof of \eqref{Phi_c}.
	\end{proof}
	
	\begin{proof}[Proof of Theorem {\rm \ref{Th_binom}}]
		We adopt the simple approach concerning the classical M\"oller-Endo formulas due to 
		Gallot et al. \cite[Sect.\ 3]{GalMorHom2011}. Also see Herrera-Poyatos and Moree \cite[Sect.\ 4]{HerMor2021}.
		
		From \eqref{Phi_gen_1_bis} we have 
		\begin{equation*}
			\Phi_{A,S,n}(x) = (-1)^{\varrho_S(n)} \prod_{d=1}^{\infty} \left(1-x^d\right)^{\mu_{A,S}(n/d)}, 
		\end{equation*}
		with the notation $\mu_{A,S}(t)=0$ if $t$ is not an integer. Writing for $|x|<1$ the Taylor series 
		expansion of $(1-x^d)^{\mu_{A,S}(n/d)}$ we deduce 
		\begin{equation*}
			\Phi_{A,S,n}(x) = (-1)^{\varrho_S(n)} \prod_{d=1}^{\infty} \sum_{j_d=0}^{\infty} (-1)^{j_d} 
			\binom{\mu_{A,S}(n/d)}{j_d}x^{dj_d}, 
		\end{equation*}
		and by identifying the coefficient of $x^k$ we obtain \eqref{form_coeff}. 
	\end{proof}

	\begin{proof}[Proof of Theorem {\rm \ref{Th_recursion}}]
		Similar to the classical case, from Vi\`ete's and Newton's formulas we deduce the recursion formula 
		\begin{equation*}
			a_{A,n}(k)= -\frac1{k} \sum_{j=1}^k a_{A,n}(k-j)c_{A,n}(j),
		\end{equation*}
		where $a_{A,n}(0)=1$.
		
		If $n=p^t$ is an $A$-primitive integer, then $A(p^t)=\{1,p^t\}$, and by the H\"older-type identity \eqref{Holder} we have   
		\begin{align*}
			c_{A,p^t}(k) & = \begin{cases} p^t-1, & \text{ if $p^t\mid k$};\\ 
				-1, & \text{ otherwise},
			\end{cases}\\
			& = \mu_A(p^t) \mu_A((k,p^t)_A)\varphi_A((k,p^t)_A).
		\end{align*}
		
		Hence, by multiplicativity, if $n$ is a product of $A$-primitive integers, namely $n=p^{t_1}\cdots p^{t_r}$, then  
		\begin{equation*}
			c_{A,n}(k)= \mu_A(n) \mu_A((k,n)_A)\varphi_A((k,n)_A),
		\end{equation*}
		holds and we deduce that
		\begin{equation*}
			a_{A,n}(k)= -\frac{\mu_A(n)}{k} \sum_{j=1}^k a_{A,n}(k-j) \mu_A((j,n)_A)\varphi_A((j,n)_A),
		\end{equation*}
		as stated.
	\end{proof}

\end{document}